\documentclass{amsart}
\usepackage[mathscr]{eucal}
\usepackage{color}
\usepackage{xypic}
\usepackage{amsfonts}   
\usepackage{amsmath}
\usepackage{amsthm}
\usepackage{amssymb}
\usepackage{latexsym}
\usepackage[english]{babel}
\usepackage[utf8]{inputenc}

\let\nc\newcommand

\nc{\la}{\label}

\newtheorem{theorem}{Theorem}[section]
\newtheorem{definition}[theorem]{Definition}
\newtheorem{corollary}[theorem]{Corollary}
\newtheorem{lemma}[theorem]{Lemma}
\newtheorem{proposition}[theorem]{Proposition}
\newtheorem{example}[theorem]{Example}
\newtheorem{remark}[theorem]{Remark}

\def\k{\mathsf k}

\newcommand{\into}{\,\,\hookrightarrow\,\,}

\def\k{\mathsf k}

\title{Generalizations of noncommutative Noether's problem}

\author{João Fernando Schwarz}

\address{CMCC, Universidade Federal do ABC, Avenida dos Estados, 5001, Santo Andre, Brazil}

\email{jfschwarz.0791@gmail.com}
\begin{document}
\maketitle

\begin{abstract}
    Noether's problem is a classical and very important problem in algebra. It is an intrinsecally interesting problem in invariant theory, but with far reaching applications in the sutdy of moduly spaces, PI-algebras, and the Inverse problem of Galois theory, among others. To obtain a noncommutative analogue of Noether's problem, one would need a significant skew field that shares a role similar to the field of ratioal functions. Given the importance of the Weyl fields due to Gelfand-Kirillov's Conjecture, in 2006 J. Alev and F. Dumas introduced what is nowdays called the noncommutative Noether's problem. Many papers in recent years \cite{FMO}, \cite{EFOS}, \cite{FS}, \cite{Tikaradze} have been dedicated to the subject. The aim of this article is to generalize the main result of \cite{FS} for more general versions of Noether's problem; and consider its analogue in prime characteristic.

    \medskip

\medskip

\medskip
\end{abstract}

\section*{Notions of rationality}
Let $\k$ denote an arbitrary base field.

The question of rationality of fields is an important one, whose study goes back to more than a century ago. One of the central problems in this area is the Lüroth problem: let $F$ be a finite extension of $\mathsf{k}$, with $\mathsf{k} \subsetneq F \subset k(x_1, \ldots, x_n)$. Is it true that, then, $F$ is also a purely transcendental extension of $\mathsf{k}$?

There is a nice geometric interpretation of this problem. Let $X$ be a variety with $\k(X)=F$. The embedding $F \subset \k(x_1,\ldots,x_n)$ induces a dominat rational map $f: \mathbb{P}^n \rightarrow X$. Varieties with this property are called unirational, and varieties birationally equivalent with a $\mathbb{P}^m$ are called rational. So geometrically Lüroth's problem asks: is every unirational variety rational?

Lüroth problem has a positive solution when $n=1$, by Lüroth's Theorem (see, e.g., \cite{Jacobson}), and when $n=2$ and the base field is algebraically closed of characteristic 0, by Castelnuovo's rationality criterion (see, e.g., \cite{Zariski}). However, if the base field is not algebraically closed, counter-examples exists already for $n=2$ even for $\mathsf{k}=\mathbb{R}$. For instance, the field of fractions of

\[ \mathbb{R}[x,y,z]/(x^2+y^2-z(z-1)(z-2)) \]

is unirational but not a purely transcendental extension \cite{Formanek}. When $n=2$, counter-examples also were found in algebraically closed fields of prime characteristic by Zariski \cite{Zariski}. When $n \geq 3$, counter-examples exists even for algebraically closed fields of zero characteristic \cite{Formanek}. 

There is still another notion of rationality that is useful: let $F$ be a finite extension of $\mathsf{k}$. If for some indeterminates, $F(x_1,\ldots,x_r)$ is a purely transcendental extension of $\k$, we say that $F$ is stably-rational. Geometrically, a variety $X$ is stably-rational if, for some $m>0$, $X \times \mathbb{P}^m$ is rational. It should be noticed that rational $\subsetneq$ stably-rational $\subsetneq$ unirational. For this and other notions of rationality, see \cite{C-T}.

\section*{Rationality of the field of invariants and Noether's problem}

There is an important subcase of Lüroth's problem, which talks about the rationality of the ring of invariants of a purely transcendental extension. It was introduced and studied by Emmy Noether \cite{Noether0}, \cite{Noether}:

\textbf{Noether's Problem:} Let $S_n$ acts by permutation of the variables in the rational field $\k(x_1,\ldots,x_n)$. Let $G<S_n$ be a subgroup that permutes transitively the variables. When $\k(x_1, \ldots, x_n)^G$ is a purely transcendental extension of $\k$ --- with necessarily the same transcendence degree? Or, in other words, when $\k(x_1, \ldots, x_n)^G \simeq \mathsf{k}(x_1,\ldots,x_n)$?

\begin{remark}
    A more appropriate name would be Noether's Conjecture, for she believed the above question to have a positive solution for all $G$.
\end{remark}

Noether introduced this question thinking in applications to the Inverse problem in Galois theory. Consider $\mathsf{k}= \mathbb{Q}$ (or any other Hilbertinian field, see, e.g. \cite[Chapter 3]{Jensen}). If Noether's Problem has a positive solution for $G$, she showed that so has the Inverse problem of Galois theory: there is a Galois extension $L$ of $\mathbb{Q}$ such that $Gal(L, \mathbb{Q})=G.$

We now introduce an useful terminology: let $G$ be a finite group, and consider the field of rational functions $\k(x_h)_{h \in G}$, with the variables indexed by the elements of $G$. There is a natural action by permutations of $G$ on $\k(x_h)_{h \in G}$: $g \in G$ sends $x_h$ to $x_{gh}$. The field of invariants $\k(x_h)_{h \in G}^G$ is denoted $\k(G)$.

Some important cases of positive solution to Noether's Problem are:

\begin{theorem}\label{lista-1}

\begin{enumerate}
    \item $S_n$ acting on the rational function field in $n$-indeterminates.
    \item $\mathcal{A}_n$ acting on the rational function field in $n$-indeterminates, for $n=3,4,5$. The case of alternating groups remains open for $n>5$ \cite{Jensen}. 
     \item $S_n$ action on $\k(x_1, \ldots, x_n,y_1,\ldots,y_n)$, permuting $x$'s and $y$'s simultaneously \cite{Miyata}
     \item The action of the quaternion group $Q_8$ in $\k(x_1,x_2,x_3,x_4)$ \cite{Jensen}
     \item Let $G$ be a $p$-group, and $char \, \k=p$. Then $\k(G)$ is a purely transcendental extension. \cite{Hoshi}
     \item $G$ is a finite abelian group with exponent $e$ such that either $char \, \k=0$ or it is coprime with $e$; and the field possesses a primitive e-th root of unity. \cite{Jensen}
     \item Let $n$ be a positive integer. Then $\mathbb{Q}(C_n)$ if and only if $n$ divides
     \[ 2^2.3^m.5^2.7^2.11.13.17.19.23.29.31.37.41.43.61.67.71 \]
$m \geq 0$. \cite{Plans}
\end{enumerate} 
\end{theorem}

The first counter-examples to Noether's problem are due to Swan \cite{Swan} and Vokresenskii \cite{Vokresenskii}, independently, over the rational numbers. Their counter-examples consisted of cyclic groups acting by permutation, the smallest one being $C_{47}$.

Later Lenstra \cite{Lenstra} classified all finite abelian groups of permutations for which Noether's problem has a positive solution; he found that the smallest group that can give a counter-example is $C_8$. Finally, Saltman \cite{Saltman} obtained the first counter-examples over algebraically closed fields.

Noether's problem continues to be extremely relevant to the Inverse problem in Galois theory \cite{Jensen}. For more information about permutation's Noether's problem, see \cite{Jensen} and \cite{Hoshi}.

\textbf{Linear Noether's problem} Let $G<GL_n(\mathsf{k})$ be a finite group acting linearly on $\mathsf{k}(x_1,\ldots,x_n)$. Is $\mathsf{k}(x_1, \ldots, x_n)^G$ a purely transcendental extension?

The first person to consider this kind of question, althought not in a sistematic way, was Burnside \cite{Burnside}.

Here is an important list of cases of positive solution:

\begin{theorem}\label{lista-2}
\begin{enumerate}
    \item All permutation actions considered previously.

    \item $n=1$, $G$ arbitrary. 
    
    \item $n=2$, $G$ arbitrary. 

    \item $n=3$, $G$ arbitrary, $\k$ algebrically closed of zero characteristic. \cite{Dumas}

    \item By Chevalley-Shephard-Todd Theorem (see Theorem \ref{C-S-T}), whenever the natural representation of $G$ is by a pseudo-reflection group and $char \, \k$ and $|G|$ are coprime. \cite{Bourbaki}, \cite{Dumas}.
\end{enumerate}
\end{theorem}

Again, as int the case of permutation groups, a positive solution of linear Noether's problem for a group $G$, shows that the Inverse problem in Galois theory has a positive solution for the same group, among other things \cite{Jensen}. In particular, this is be far the easiest way to show that pseudo-reflection groups give a positive solution to the Inverse problem.

Linear Noether's problem and the original Noether's problem are linked by the no-name lemma \cite{Plans0}: Let $G$ be a finite group. There exists a faithful, finite-dimensional, linear $\k$-representation $G \into GL(V)$ such that $\k(V)^G$ is a purely transcendental extension if and only if $\k(G)$ is stably-rational.

When we allow the action of infinite groups; that is, rational representations of a linear connected algebraic group $G$ on a finite dimensional vector space $V$, the questions of the rationality of $V/G$ or $P(V)/G$ --- which can also be understood as the rationality of the invariants of $k(V), \, k(P(V))$ --- are related to the question of rationality of many moduli spaces \cite{Bohning}.
As an example, if $G$ is a connected solvable algebraic group over an algebraically closed field, and $V$ is a rational representation, $V/G$ is rational --- the action of $G$ stailizes a flag of $V$ (by Lie-Kolchin theorem) and we may apply Miyata's theorem \cite{Miyata}. This problem, in this generality, has also many other important apllications: see \cite{C-T} and \cite{Dolgachev}.

Now let's recall the following notion

\begin{definition}
A $G$-lattice is a faithful $G$-module $M$ which is a finitely generated free abelian group.
\end{definition}

\textbf{Multiplicative Noether's Problem:} Let $M$ be a $G$-lattice. The group algebra $\mathsf{k}[M]$ is a ring of Laurient polynomials in $rank \, M$ indeterminates, where $G$ acts by algebra automorphisms. When is the invariant subfield $\mathsf{k}(M)^G$ a purely transcendental extension?

Pherhaps the most spetacular application of multiplicative Noether's problem is due to Procesi \cite{Procesi}. He realized that the question of (stable)-rationality of the center of the division ring of fractions of the ring of $2$ $n \times n$ generic matrices is equivalent to the positive solution of particular case o multiplicaive Noether's problem. A nice simplification of his ideas can be found in \cite{Formanek}. The center is a purely trancendental extension for $n=2,3,4$. For greater values of $n$, it is still an open probelm \cite{Drensky}.

There is a multiplicative analogue of the Chevalley-Shephard-Todd Thereom (see Theorem \ref{C-S-T} below ).

\begin{definition}
    Let $L$ be a $G$-lattice. We say that $G$ is a reflection group if, with its induced action on the $\mathbb{Q}$-vector space $\mathbb{Q} \otimes_\mathbb{Z} L$, $G$ acts as a reflection group.
\end{definition}

\begin{theorem}\label{multiplicative-cst}
    \emph{(Multiplicative Version of Chevalley-Shephard-Todd Theorem)}

    Let $L$ be a $G$-lattice, with $char \, \k$ and $|G|$ coprime. The following are equivalent:

    \begin{enumerate}
        \item $k[L]^G$ is a regular ring.
        \item $k[L]$ is a finitely generated projective $k[L]^G$-module.
        \item $k[L]$ is a finitely generated free $k[L]^G$-module.
        \item $k[L]^G$ is mixed Laurient polynomial ring.
        \item $G$ acts as a reflection group on $L$ and $\mathbb{Z}[L]^G$ is a unique factorization domain.
    \end{enumerate}
\end{theorem}
\begin{proof}
  \cite[Theorem 7.1.1.]{Lorenz}
\end{proof}

\begin{corollary}\label{example-semilaurient}
    If $L$ is a $G$-lattice that satisfy any of the above conditions, then $\mathsf{k}(L)^G$ is a purely transcendental extension. In particular, $\k[L]^G$ is polynomial if and only if $L$ is isomorphic to the weight lattice of some reduced root system and $G$ is the Weyl group.
\end{corollary}
\begin{proof}
    \cite[Corollary 7.1.2]{Lorenz}.
\end{proof}

Apart from this result, we remark that, in practice, is very difficult to find groups $G$ that satisfy the conditions of Theorem \ref{multiplicative-cst}, and in general one needs computer assistance \cite{Lorenz0}. Nonetheless

\begin{theorem}\label{lista-3}
If $rank \, L$  of the $G$-lattice is $1 \, ,2$ or $3$, $\k(L)^G$ will always be a purely transcendental extension, independtly of $G$.
\end{theorem}
\begin{proof}
\cite{Hoshi}
\end{proof}

More on multiplicative Noether's problem can be found in the survey \cite{Hoshi} and on the book \cite{Lorenz}.

In this paper we will consider the most general form for Noether's problem. We will use the terminology in \cite{Jensen}:

\textbf{General Noether's problem }Let $G$ be any finite subgroup of automorphisms of $\k(x_1,\ldots,x_n)$ whatsoever. When is $\k(x_1,\ldots,x_n)^G$ a purely transcendental extension?

By Lüroth's Theorem, this problem has a positive solution for any field when $n=1$, and by Castelnuovo rationality criterion, for $n=2$ when $\k$ is algebraically closed and $char \, \k=0$

\section*{Noncommutative Noether's problem}

In the paper \cite{Alev} J. Alev and F. Dumas introduced the following noncommutative analogue of Noether's problem, usually called just noncommutative Noether's problem. Let $char \, \k=0$, $A_n(\k)$ denote the Weyl algebras, and $F_n(\k)$ denote their skew field of fractions, the Weyl fields. More generally, We write $A_{n,s}(\k)$ for $A_n(\k(x_1, \ldots, x_s))$, and $F_{n,s}(\k)$ their skew field of fractions. For the sake of simplicity write $F_{0,s}(\k)=\k(x_1, \ldots, x_s)$.

The point of view of noncommutative Noether's problem was influenced by the Gelfand-Kirillov conjecture \cite{GK}, that had great influence in the study of enveloping algebras, but was eventually shown to be false in general. The Conecture said that given be an algebraic Lie algebra $\mathfrak{g}$, then the skew field of fractions of $U(\mathfrak{g})$ is isomorphic to $F_{n,s}(\k)$ for adequate $n,s$. So the Weyl fields can be considered good noncommutative analogues of the field of rational functions. For more on the Conjecture, see \cite{Premet}.

\textbf{Noncommutative Noether's problem} Let $G$ be a finite group acting linearly in $A_n(\k)$, and hence on $F_n(\k)$. When $F_n(\k)^G \simeq F_n(\k)$?

\begin{remark}
    One might wonder why we don't consider the possibilty that $A_n(\k)^G$ is isomorphic to $A_n(\k)$. By a result of Alev and Polo \cite{AP}, this is knwon to be impossible.
\end{remark}

\begin{remark}
The question for infinite $G$ was also considered. In this case a small modification on the statemente is necessary; see \cite{Alev}.
\end{remark}

\begin{remark}
Invariants of the first Weyl field under other actions were considered in \cite{Alev0}.    
\end{remark}

In \cite{Alev} the problem was shown to have positive solution when $n=1,2$ and when the natural action of $G$ decomposes in a direct sum of one dimensional $G$-modules. In \cite{FMO} the problem was solved for the symmetric group with its permutation action, and as an consequence it was obtained the analogue of Gelfand-Kirillov Conjecture for finite $W$-algebras of type $A$. In \cite{EFOS} this fact was generalized for all complex reflection groups, and as corollary it was obtained the analogue of the Gelfand-Kirillov Conjecture for spherical subalgebras of rational Cherednik algebras and linear Galois algebras (which includes $U(\mathfrak{gl}_n)$, finite $W$-algebras of type $A$, OGZ algebra, \cite{FS2}). In \cite{FS} it was proved that if a finite linear action gives a positive solution to linear Noether's problem, then the same action gives a positive solution of noncommutative Noether's problem. This result was used to show the validity of Gelfand-Kirillov Conjecture for spherical subalgebras of trigonometric Cherednik algebras in \cite{Schwarz} and certain algebras in the paper \cite{Erich}. In \cite{Tikaradze} a kind of the converse result was shown: if $\k=\mathbb{C}$, and $G$ is a finite group of linear automorphsism defined over $\mathbb{Z}$ and $F_n(\mathbb{C})^G \simeq F_n(\mathbb{C})$, then for algebraically closed fields $\k$ of characteristic big enough, $\k(x_1,\ldots,x_n)^G$ is stably rational, and with this result counter-examples to noncommutative Noether's problem were found --- in fact, the group actions were the same as Saltman's counter-examples to Noether's problem, for in these counter-examples, the invariant subfield is also not stably-rational. It remains an interesting open problem to determine if linear Noether's problem and its noncommutative analogue are equivalent, although we conjecture that this does not hold.

Our purpose in this paper is two-fold. The first is to generalize the noncommutative Noether's problem, and the second is to consider its version in prime characteristic. We will denote the usual (Grothendieck's) differential operator ring by $\mathcal{D}$ and its crystalline version (from \cite{BMR}) by $\mathcal{D}_c$.

The main theorem of \cite{FS} is Theorem 1.1: if $G<GL_n(\k)$ is finite group of automorphisms that act linearly on $\k(x_1, \ldots, x_n)$ and on the Weyl algebra $A_n(\k)$ in the same way, a positive solution to linear Noether's problem implies a positive solution to noncommutative Noether's problem: $F_n(\k)^G \simeq F_n(\k)$. This theorem was proven using purely ring theoretical methods.

To generalize this result, we will need some basic affine algebraic geometry. Hence, from now on, we assume $\k$ algebraically closed.

Following the terminology in \cite{McConnell}, we will call $B_n(\k)=\mathcal{D}(\k(x_1,\ldots,x_n))=k(x_1,\ldots,x_n)\langle \partial_1, \ldots, \partial_n \rangle$. Any finite group $G$ of automorphisms of $\k(x_1,\ldots,x_n)$ extends to a group of automorphisms of $B_n(\k)$, and hence of $F_n(\k)$

\begin{theorem}\label{new-1}
In the context above, if general Noether's problem has a positive solution, then $F_n(\k)^G \simeq F_n(\k)$    
\end{theorem}

We will offer also a direct, much simpler, proof of the next theorem, even it being just a corollary of the previous one.

\begin{theorem}\label{new-2}
Let $M$ be an $G$-lattice. Identify $\k[M]$ with $\k[x_1^{\pm 1}, \ldots, x_n^{\pm 1}]=\mathcal{O}(\k^{\times n})$, where $\k^{\times n}$ is the $n$-torus. Then $G$ acts on $\mathcal{D}(\k^{\times n})$, and hence on $F_n(\k)$. If multiplicative Noether's problem has a positive solution, then $F_n(\k)^G \simeq F_n(\k)$
\end{theorem}

Now we move to prime characteristic (and keep the field algebraically closed). The definition of the Weyl algebra by generators and relations make perfect sense in prime characteristic, and $A_n(\k)$ is still a Noetherian domain (for more about the Weyl algebra in prime characteristic, see \cite{Revoy}. In particular, it is an Azumaya algebra over its center). In particular, its skew field of fractions, the Weyl fields $F_n(\k)$, exists, although now they are finite dimensonal over their centers. It has been stated explictly in this author PhD thesis that he conjectured some form of noncommutative Noether's problem would make sense in prime characteristic.

However, Grothendieck's rings of differential operators in prime characteristic are not suitable for us, as they are not Noetherian or domains, and in particular, $A_n(\k)$ and $\mathcal{D}(\mathbb{A}^n)$ are very different rings.

On the contrary, rings of crystalline differential operators have the desired properties: if $X$ is an smooth affine variety, $\mathcal{D}_c(X)$ is a Noetherian domain \cite{BMR}. We also have $A_n(\k)=\mathcal{D}_c(\mathbb{A}^n)$

In \cite[Theorem 1.2]{FS} we proved that given two affine varieties $X, Y$, and a finite group $G$ of automorphisms of $X$, if $X/G$ is birationally equivalent to $Y$, then $Frac \, \mathcal{D}(X)^G \simeq Frac \, \mathcal{D}(Y)$.

We have an analogue for rings of crystalline differential operators.

\begin{theorem}\label{new-3}
Let $\k$ be an algebraically closed field of prime characteristic. If $X$ is an smooth affine variety and $G$ a finite group of automorphisms of it, such that $X/G$ is birationally equivalent to an affine smooth variety $Y$, then $Frac \, \mathcal{D}_c(X)^G \simeq Frac \, \mathcal{D}_c(Y)$.
    
\end{theorem}

\begin{remark}
Notice that, unlike the characteristic $0$ case, in prime characteristic we must restrict attention to smooth varieties.    
\end{remark}

\begin{corollary}\label{new-4}
    Let $G$ be a finite group of automorphisms of $\k[x_1, \ldots, x_n]$. If $\k(x_1, \ldots, x_n)^G$ is a purely transcendental extension, then $F_n(\k)^G \simeq F_n(\k)$
\end{corollary}

So, just like in characteristic 0 case, in prime characteristic a positive solution to Noether's problem implies a positive solution to its noncommutative analogue.

We also have

\begin{corollary}\label{torus}
Let $G$ be a finite group of automorphisms of $\k[x_1^{\pm 1}, \ldots, x_n^{\pm 1}]$.If $\k(x_1, \ldots, x_n)^G$ is a purely transcendental extension, then $F_n(\k)^G \simeq F_n(\k)$.
\end{corollary}

This last corollary includes multiplicative Noether's problem.

Our next result is the version in prime charcteristic on the Gelfand-Kirillov Conjecture for rational Cherednik algebras (rational Cherednik algebras in prime characteristic were studied on a number of places, for instance, \cite{BC}).

\begin{theorem}\label{new-5}
 Assume $2|W| \in \k^\times$. Let $U_{1,c}(h,W)$ be a spherical subalgebra of a rational Cherednik algebra. Then $Frac \, U_{1,c}(h,W) \simeq F_n(\k) $, where $n=dim \, h$
\end{theorem}

We also have a generalization of \cite[Theorem 3.14]{Schwarz} to prime characteristic.

\begin{theorem}\label{new-6}
Let $\k$ be an algebraically closed field of prime characteristic. Let $X$ be an smooth affine variety and $G$ a finite group of automorphisms of it such that $X/G$ is birrationally equivalent to a smooth affine variety $Y$. Then $\mathcal{O}(T^*X)^G$ has a field of fractions which is isomorphic as a Poisson field to the fraction field of $\mathcal{O}(T^*Y)$.
\end{theorem}

With this result, we can consider J. Baudry's Poisson Noether's problem \cite{Baudry} in prime characteristic.

Let $(V, \omega)$ be a symplectic vector space, $n=dim \, V$ and call its Poisson function field as $\mathcal{P}_n(\k)$. Let $X$ be a Poisson variety. We call $X$ Poisson rational if $\k(X)$ is isomorphic to $\mathcal{P}_n(\k), \, n=dim \, X$, as a Poisson field. Hence Poisson rationality is a refinement of the notion of rationality in the class of Poisson varieties.

 Poisson Noether's problem asks: let $(V, \omega)$ a symplectic vector space and $G$ a finite group of symplectomorphisms. When is $V/G$ Poisson rational?

Our contribution to this problem in prime characteristic will be Theorem \ref{prime-Baudry}.

\begin{remark}
    Other noncommutative analogues of Noether's problem can be found in \cite{KPPV} for free skew fields, and in \cite{Fryer}, \cite{FH}, \cite{H} for the skew field of tensor products of quantum planes.
\end{remark}

\section*{Rationaliy of rings of differential operators}

In this section the base field is algebraically closed of zero characteristic.

\begin{definition}

Let $X$ be an irreducible variety.

\begin{enumerate}
    \item $\mathcal{D}(X)$ is called rational if $Frac \, \mathcal{D}(X) \simeq F_m(\k)$ for some $m$.

    \item $\mathcal{D}(X)$ is called stably-rational if there exists an $n$ and $m$ such that $Frac \, (\mathcal{D}(X) \otimes A_n(\k)) \simeq F_m(k)$.

    \item $\mathcal{D}(X)$ is called unirational if there is an embedding $\mathcal{D}(X)$ into some $F_m(\k)$, for adequate $m$.
\end{enumerate}

\end{definition}

\begin{theorem}\label{new-7}
    \begin{enumerate}
        \item If $X$ is rational, then $\mathcal{D}(X)$ is rational. More precisely, if $X$ is birationally equivalent to $\mathbb{A}^n$, $Frac \, \mathcal{D}(X) \simeq F_n(\k)$.

        \item If $X$ is stably rational , then $\mathcal{D}(X)$ is stably rational. More precisely, if $X \times \mathbb{A}^n$ is rational, then $Frac \, (\mathcal{D}(X) \otimes A_n(\k)) \simeq F_m(\k)$, with $m=dim \, X + n$

        \item if $X$ is unirational, then $\mathcal{D}(X)$ is unirational. More precisely, if there is a rational dominant map $\mathbb{A}^n \rightarrow X$, then we have and embedding of $Frac \, \mathcal{D}(X)$ into $F_n(\k)$. If $Frac \, \mathcal{D}(X)$ embedds into $F_m(k)$, then necessarily $m \geq dim \, X$
    \end{enumerate}
\end{theorem}

\section{Preliminaries}

Suppose initially that $\k$ is an arbitrary field. Lets recall some definitions.

\begin{definition}
The $n$-th Weyl algebra $A_n(\k)$ is the algebra generated by generators, $x_1,\ldots,x_n,y_1,\ldots,y_n$ subject to relations

\[ [x_i,x_j]=[y_i,y_j]=0, \, [y_i,x_j]=\delta_{ij}, i,j=1,\ldots, n\]
\end{definition}

In case $char \, \k=0$, the Weyl algebras are simple Noetherian domains with center $\k$ \cite{McConnell}. In case of positive characteristic, $A_n(\k)$ again is Noetherian domain, but now the algebra is a finite module over its center $\k[x_1^p,\ldots,x_n^p,y_1^p,\ldots,x_n^p]$ \cite{Revoy}.

The Weyl algebras are related to rings of differential operators, introduced by Grothendieck \cite{Grothendieck}:

\begin{definition}
Let $A$ be a $\k$-algebra. Set

\[\mathcal{D}_0(A)=\{\theta \in End_\k \, A| [\theta, a]=0, \forall a \in A\} \simeq A,\]

and, inductively,

\[ \mathcal{D}_i(A)=\{\theta \in End_\k \, A| [\theta, a] \in \mathcal{D}_{i-1}(A), \forall a \in A \}. \]

The ring of differential operators on $A$ is $\mathcal{D}(A)=\bigcup_{i=0}^\infty \mathcal{D}_i(A)$.
\end{definition}

For an affine variety $X$, we call the ring of differential operators in $X$, $\mathcal{D}(X)=\mathcal{D}(\mathcal{O}(X))$. In case $char \, \k =0$, $\mathcal{D}(X)$ is an Ore domain \cite{FS}, and if moreover $X$ is smooth, $\mathcal{D}(X)$ is simple and left and right Noetherian \cite{McConnell}. Moreover, by the definition, $\mathcal{D}(X)$ comes with a filtration such that the associated graded algebra is $\mathcal{O}(T^* X)$ (for smooth $X$) (\cite{Grothendieck}, \cite{McConnell}). We have $A_n(\k)=\mathcal{D}(\mathbb{A}^n)$.

In caracteristic 0 and affine regular domains $A$, there is an alternative definition of differential operators that coincides with the Grothendieck's one: $\mathcal{D}(A)$ can be described as the subalgebra of $End_\k \, A$ generated by $A$ and $Der_\k \, A$ \cite[15.5.5]{McConnell}. For arbitrary $A$, if we denote this definition as $\Delta(A)$, we have that just htat$\Delta(A) \subset \mathcal{D}(A)$, and the inclusion may be proper if $A$ is not regular. In fact, the famous Nakai Conjecture \cite{McConnell} says that $\Delta{A}=\mathcal{D}(A)$ if and only if $A$ is a finitely generated regular domain. Notice that some patological behavior can happen if $A$ is not finitely generated. In a recent preprint \cite{MS}, an example of regular domain $A$ but not finitely generated is given such that $Der_\k \, A=0$ and the natural inclusion map $A \rightarrow \mathcal{D}(A)$ is an isomorphism.

For smooth varieties in $char \, \k=p>0$, as shown first by Smith \cite{Smith0}, the correct definition of ring of differential operators is the Grothedieck's one.  However, the situation changes drastically in positive characteristic. For instance, denoting by $P_n$ the polynomial algebra in $n$-indeterminates, $\mathcal{D}(P_n)$ is not Noetherian, not finitely generated, and has a lot of zero-divisors and nilpotent elements. Its Gelfand-Kirillov dimension also gives the ``wrong" number: $n$, instead of $2n$. However, the algebra is still simple \cite{Bavula}. In particular, $A_n(\k)$ and $\mathcal{D}(\mathbb{A}^n)$ are non-isomorphic rings.

It is clear that the classical notion of differential operator will lead us nowhere in our task to considering the noncommutative Noether's problem in prime characteristic --- they are not even domains. The solution is found in changing the technology: we will need a modified version of differential operators introduced in \cite{BMR} (a particular case of a more general definition in \cite{BO}): crystaline differential operators.

\begin{definition}
Let $\k$ be an algebraically closed field of prime characteristic, $X$ an smooth affine variety. $\mathcal{D}_c(X)$, the ring of crystalline differential operators on $X$, is generated by $\mathcal{O}(X)$ and $Der_\k \mathcal{O}(X)$ subject to the relations

\[ f.\partial=f\partial, \, \partial.f - f.\partial=\partial(f), \]

\[ \partial.\partial'-\partial'.\partial=[\partial,\partial'], f \in \mathcal{O}(X), \partial,\partial' \in Der_\k \mathcal{O}(X).\]

\end{definition}

\begin{remark}
The notion of crystalline differential operators works only for smooth varieties.    
\end{remark}

We have now that $\mathcal{D}_c(\mathbb{A}^n) \simeq A_n(\k)$.

$\mathcal{D}_c(X)$ has a natural filtration: $\mathcal{D}_c^0(X)=\mathcal{O}(X), \, \mathcal{D}_c^i(X)=\mathcal{D}_c^i{i-1}(X)+Der_\k \mathcal{O}(X).\mathcal{D}_c^i{i-1}(X)$.

\begin{proposition}\label{filtered}\cite{BMR}
\begin{enumerate}
\item $gr \, \mathcal{D}_c(X)=\mathcal{O}(T^*X)$.
\item The Poisson algebra structure on $\mathcal{O}(T^*X)$ induced by the filtered quantization by $\mathcal{D}_c(X)$ coincides with the usual Poisson algebra structure from the standard symplectic form on $T^*X$.
\end{enumerate}
\end{proposition}

Since the associated graded algebra is a finitely generated Noetherian domain, by usual filtered techniques \cite{McConnell}, we have that $\mathcal{D}_c(X)$ is a finitely generated left and right Noetherian domain. Hence we can study skew fields of fractions.

Moreover, we have (\cite{McConnell}):

\begin{proposition}\label{dimensions}
$GK\, \mathcal{D}_c(X)=2dim \, X$, $gl.dim \, \mathcal{D}_c(X) \leq dim \, X$, $\mathcal{K} \, \mathcal{D}_c(X) \leq dim \, X$.
\end{proposition}

\begin{remark}
Proposition \ref{filtered} is well known for fields of zero characteristic and usual differential operators.    
\end{remark}

\section{Proof of characteristic 0 results}

Before we move to questions about the caracteristic prime case, let's discuss the results we have in characteristic 0.

\subsection{Generalizations of noncommuative Noether's problem}

In this section we will work over algebraically closed fields.

We want to use this result to prove Theorem \ref{new-1}. However, our automorphism group comes from $\k(x_1,\ldots,x_n)$, and not from an affine variety --- so we can't use \cite[Theorem 1.2]{FS}. But we are going to show that there always exists an open affine variety $X \subset \mathbb{A}^n$ such that the group of birational automorphisms on $\mathbb{A}^n$ restricts to biregular automorphism on $X$, and $X/G$ is still rational. Clearly, then, $Frac \, \mathcal{O}(X)^G \simeq \k(x_1,\ldots,x_n)^G \simeq \k(x_1, \ldots, x_n)$

\begin{lemma}\label{main-lemma}
Let $G$ be a finite group of birational automorphisms of an affine variety $X$. Then there is a $G$-invariant affine open subset $U$ where $G$ acts by biregular automorphisms.
\end{lemma}
\begin{proof}
For each $g \in G$ there is an affine open subset $U_g \subset X$ such that $g:U_g \rightarrow g(U_g)$ is a biregular isomorphism.

Calling $U=\bigcap_{g \in G}U_g$, we have that $g: U \rightarrow g(U)$ is a biregular isomorphism for every $g \in G$. $U$ is again affine, since finite intersection of affine open subsets is again affine.

Setting $W=\bigcap_{g \in G} g(U)$, we have that $W$ is $G$-invariant, affine, and each $g \in G$ is a biregular isomorphism of $W$ onto itself.
\end{proof}

\begin{example}\label{example-1}
Let $G<GL_n(\mathbb{Z})$ act on $\k(x_1,\ldots,x_n)$ as follows:

\[ g.x_j=\prod_{i=1}^n x_i^{a_{ij}}, g=(a_
{ij})_{i,j=1,\ldots,n} \in GL_n(\mathbb{Z}), \, j=1,\ldots,n.\]

These are precisely the group actions that arise in the multiplicative Noether's problem mentioned in the Introduction. $G$ corresponds to a finite group of birational automorphisms of the affine space $\mathbb{A}^n$. A open affine subset $U$ as in the previous lemma is $Spec \, \k[x_1^{\pm 1},\ldots,x_n^{\pm 1}] = \k^{\times n}$, the $n$ torus. This is the case because $G$ acts on $\k[x_1^{\pm 1},\ldots,x_n^{\pm 1}]$ by \emph{algebra} automorphisms and on $\k^{\times n}$ by biregular automorphisms.
\end{example}

\begin{lemma}\label{step-1}
Suppose $X$ an affine variety, $G$ a group acting on it, such that $X/G$ is rational.  Then  $Frac \, \mathcal{D}(X/G) \simeq F_n(\k)$, $n=dim \, X$.
\end{lemma}
\begin{proof}
Let $S=\mathcal{O}(X)^G \setminus \{ 0 \}$. $Frac \, \mathcal{D}(X/G)= Frac \, \mathcal{D}(X/G)_S =Frac \, \mathcal{D}(\mathcal{O}^G_S)$, where we used \cite[Proposition 1.8]{Muhasky} and that $\mathcal{O}(X/G)=\mathcal{O}(X)^G$. Since $X/G$ is rational, $\mathcal{O}(X)^G_S \simeq \k(x_1,\ldots,x_n)$. Hence $Frac \, \mathcal{D}(X/G) \simeq Frac \, B_n(\k)=F_n(\k)$. 
\end{proof}

The proof of the following lemma work in any characteristic.

\begin{lemma}\label{blaster-lemma}

Let $G$ a finite group acting on an affine variety $X$. Then there is an open subset $U \subset X$ where $G$ restricts to a free action. Moreover, there is no loss of generality in assuming $U$ affine.
    
\end{lemma}
\begin{proof}
    For each $g \neq 1$ in $G$ call $Y_g=\{ x \in X| g(x)=x \}$. $Y_g$ is closed being the inverse, uder the map $(id, g): X \rightarrow X \times X$, of the diagonal. Since $G$ is finite, $Y= \bigcup Y_g$ is a closed subset of $X$ in the Zarisk topology, and $U = X \setminus Y$ is an open set where $G$ acts freely.
\end{proof}

\begin{theorem}\label{free0}
    Let $X$ be a variety and $G$ a group acting freely on it. Then $\mathcal{D}(X)^G \simeq \mathcal{D}(X/G)$
\end{theorem}
\begin{proof}
    \cite[Theorem 3.7(1)]{CH}
\end{proof}

Now we return to our original situation. We have a group $G$ of birational automorphisms of $\mathbb{A}^n$ such that $\k(x_1,\ldots,x_n)^G \simeq \k(x_1,\ldots, x_n)$.

By Lemmas \ref{main-lemma} and \ref{blaster-lemma}, there is an open affine subset $U \subset \mathbb{A}^n$ such that $G|_U$ acts freely by biregular automorphisms on $U$, and $U/G$ is rational.

We have $F_n(\k)^G= Frac \, \mathcal{D}(U)^G$. By Theorem \ref{free0}, $Frac \, \mathcal{D}(U)^G \simeq Frac \, \mathcal{D}(U/G)$, which by Lemma \ref{step-1}, is isomorphic to $F_n(\k)$. Hence Theorem \ref{new-1} is proved.

As promised, we now give a simpler proof for Theorem \ref{new-2}. By hypothesis, we have a finite group $G$ of biregular automorphisms of the torus $\k^{\times n}$, such that $\k^{\times n}/G$ is rational. By \cite[Theorem 1.2]{FS}, $Frac \, \mathcal{D}(\k^{\times n})^G=F_n(\k)^G \simeq F_n(\k)$.
\subsection{Rationality for rings of differential operators}

Let's prove items 1, 2 and 3 of Theorem \ref{new-7}.

First we recall a well-known fact about rings of differential operators, which can be found, for instance, in \cite{Muhasky}

\begin{proposition}\label{proposition-Muhasky}
Let $A$ a finitely generated algebra which is also a domain. Let $S$ be a multiplicatively closed subset of $A$. Then the localizations of $\mathcal{D}(A)$ in the left and on the righ by $S$ exists, and they are isomorphic to $\mathcal{D}(A)_S$
    
\end{proposition}

\begin{lemma}\label{obvious-lemma}

Let $A$ be an Ore domain and $S$ any denominator set on $A$. $Frac \, A= Frac \, A_S$.
\end{lemma}
\begin{proof}
    \textbf{Proof of item 1}

    If $X$ is rational, call by $S$ its set of non-null elements, and by $A$ its ring of regular functions. Clearly $A_S \simeq \k(x_1,\ldots, x_n)$.
    , where $n$ is the Krull dimension of $A$. $Frac \, \mathcal{D}(A) = Frac \, \mathcal{D}(A)_S = Frac \, \mathcal{D}(A_S) = Frac \, \mathcal{D}(\k(x_1, \ldots, x_n)) = Frac \, B_n(\k)= Frac \, F_n(\k)$. 
\end{proof}

Let's know prove item 3.

\begin{proof}
    \textbf{Proof of item 3} If $X$ is unirational, we have an embedding of its function field $F$ into $\k(x_1,\ldots,x_n)$, where $n \geq dim X$ by transcendence degree considerations. Taking differential operators preserve the embedding $\mathcal{D}(F) \subset B_n(\k)$.

     Using lemma \ref{obvious-lemma},  $Frac \, \mathcal{D}(X) = Frac \, \mathcal{D}(F) \subset F_n(\k)$, as taking the skew-field of fractions preserves embeddings. Finally, if $Frac \, \mathcal{D}(X)$ embedds in some $F_n(\k)$, by \cite[Theorem 10]{FSS}, $dim \, X \leq n$.
\end{proof}

The hardest one to prove is the second item. So we will need some preparation.

\begin{proposition}
Let $A$ be an affine integral domain. There exists an element $c$ in $A$ such that localization $A_c$ is regular
\end{proposition}
\begin{proof}
    \cite[15.2.10]{McConnell}

\end{proof}

Let $X$ be an affine variety, and $A=\mathcal{O}(X)$. Let $c$ be one of the elements $c$ of the previous proposition such that $A_c$ is regular. I will denote by $X_c = Spec \ A_c$; $X_c$ is a smooth open affine subvariety of $X$.

\begin{proof}

\textbf{Proof of item 2}
    Let $ c \in \mathcal{O}(X)$ be such that $X_c$ is smooth. First notice that we can localize $\mathcal{D}(X) \otimes A_n(\k)$ by $c \otimes 1$ beucase this element act ad-nilpotently \cite[Thm 4.9]{KL}. So $\mathcal{D}(X_c) \otimes A_n(\k)$ is a localization of of $\mathcal{D}(X) \otimes A_n(\k)$. Suppose that $X \times \mathbb{A}^n$ is birational, say, to $\mathbb{A}^m$.; then the same holds for $X_c \otimes \mathbb{A}^n$. By transcendence degree considerations, $dim \, X + n = m$. Both $X_c$ and $\mathbb{A}^n$ are smooth affine varieties with finite Krull dimension, so using \cite[Lemma 2.5]{BO}, we have that $\mathcal{D}(X_c) \otimes A_n(\k)$ is isomorphic to $\mathcal{D}(X_c \times \mathbb{A}^n)$. Using Theorem \ref{new-7} item 1, we obtain that $ Frac \, \mathcal{D}(X_c \times \mathbb{A}^n)=F_m(\k)$. In fact, $m$ must be $dim \, X + n$, by \cite[Theorem 8]{FSS}.
\end{proof}

\section{The situation in prime
characetristic}

In this section $\k$ is an algebraically closed field of prime characteristic.

\begin{proposition}\label{localization}
Let $X$ and $Y$ be smooth affine varieties with $\k(X) \simeq \k(Y)$. Then $Frac \, \mathcal{D}_c(X) \simeq Frac \, \mathcal{D}_c(Y)$.
\end{proposition}
\begin{proof}
 Analogous to the case of zero characteristic (see remark below). The ring of crystalline differential operators sheafifies \cite{BMR}; that is, just as usual rings of differential operators, it is compatible with localization.
\end{proof}

\begin{remark}
For characteristic $0$ and usual rings of differential operators the above is \cite[15.1.25]{McConnell}.    
\end{remark}

If $G$ is a finite group of automorphisms of a smooth variety $X$ then we can define an action on $\mathcal{D}_c(X)$ by defining the action on the generators as follows:

\[ g.f=g(f), \, g.\partial=g \circ \partial \circ g^{-1}, g \in G, f \in \mathcal{O}(X), \partial \in Der_\k \mathcal{O}(X).\]

\begin{theorem}\label{free}
Let $X$ be an smooth affine variety and $G$ be a finite groups of automorphisms of $X$ that acts freely. Then

\[\mathcal{D}_c(X)^G \simeq \mathcal{D}_c(X/G)\]
\end{theorem}
\begin{proof}
 Since $G$ acts freely on $X$, the projection $\pi: X \rightarrow X/G$ is étale (\cite{Mumford}, $\oint$II.7). Hence the induced map on the tangent bundles $d \pi: T_X \rightarrow \pi^* T_{X/G}$ is an isomoprhism. Hence we also have that $Sym_{\mathcal{O}(X)} T_X \simeq \pi^* Sym_{\mathcal{O}(X/G)} T_{X/G}$ and clearly we have $\mathcal{O}(X)^G \simeq \mathcal{O}(X/G)$. By the PBW theorem (which is just Proposition \ref{filtered}), the result follows.
\end{proof}

\begin{remark}
Notice that the analogue result holds in characteristic zero: Theorem \ref{free0}.    
\end{remark}

\textbf{Proof of Theorem} \ref{new-3}
Since $G$ is finite, there is an open subset $U$ of $X$ such that the restriction of the action of $G$ to $U$ is free. Without loss of generality, we can assume $U$ affine.

Hence by Theorem \ref{free}:

\[ (1) \, \mathcal{D}_c(U)^G \simeq \mathcal{D}_c(U/G) .\]

Now, $\k(X)=\k(U)$, hence by Proposition \ref{localization} $(2) \, Frac \mathcal{D}_c(X)^G=Frac \, \mathcal{D}_c(U)^G$. Also $\k(U/G)=k(Y)$ (as $X/G$ and so $U/G$ is birationally equivalent to $Y$). So again by Proposition \ref{localization} $(3) \, Frac \, \mathcal{D}_c(U/G) \simeq Frac \, \mathcal{D}_c(Y)$. Combining (1), (2), (3) we have
\[Frac \, \mathcal{D}_c(X)^G \simeq Frac \, \mathcal{D}_c(Y) \] as desired.

I repeat now the corollary that says, essentially: Noethers problem implies noncommutative Noether's problem in prime characteristc
\begin{corollary}
    Let $G$ be a finite group of automorphisms of $\k[x_1, \ldots, x_n]$. If $\k(x_1, \ldots, x_n)^G$ is a purely transcendental extension, then $F_n(\k)^G \simeq F_n(\k)$
\end{corollary}

\begin{remark}
    We did not impose that $G$ acts linearly.
\end{remark}

\subsection{Rational Cherednik algebras in prime characterisc}
The theory of rational Cherednik algebras over fields of prime characteristic parallel the one over $\mathbb{C}$ developed in \cite{EG}. We follow \cite{BC}.

Let $\k$ be an algebraically closed field of odd characteristic. Let $(V,\omega)$ be an even dimensional vector space $V$ with a non-degenerated sympletic form $\omega$. A finite subgroup of $SP(V)$ is called a symplectic reflection group if it is generated by symplectic reflections, which are symplectic isomorphisms $g$ such that $rank \, 1-g=2$.

Let $\Gamma$ be a sympectic reflection group, $S$ the set of symplectic reflections, $c: S \rightarrow \k$ invariant under conjugation. Let $t \in \k$. We allways assume 

\[ 2|\Gamma| \in \k^\times .\]

The symplectic reflection algebra $H_{t,c}$ is the quotient of $T(V)*\Gamma$ by the relations

\[ [x,y]=t\omega(x,y)-\sum_{s \in S} c(s) \omega_s(x,y), \]

where $x,y \in V$ and $\omega_s$ is the skew-symmetric form with radical $ker(I-s)$ and coincides with $\omega$ in $Im(I-s)$.

Let $W$ be a pseudo-reflection group with representation $h$. $W$ acts on $V=h \oplus h^*$, which has a $W$-invariant symplectic form $\omega((u,f),(x,g))=g(u)-f(x)$. $W$ becomes in this way a symplectic reflection group, with the set of pseudo-reflections corresponding to the symplectic reflections. Let $0 \neq t \in \k$, $c: S \rightarrow \k$ invariant under conjugation. The symplectic reflection algebra $H_{t,c}$ is then called a rational Cherednik algebra and denoted $H_{t,c}(h,W)$.

Let $e=1/|W| \sum_{w \in W}$ be an indempotent in $H_{t,c}(h,W)$. The spherical subalgebra $U_{t,c}(h,W)$ is $eH_{t,c}(h,W)e$ (with unit $e$).

\begin{proposition}\label{rat-alg-p}
The spherical subalgebra is a finitely generated Noetherian domain \cite[Theorem 3.1]{BC}.
\end{proposition}

\begin{proof}
\textbf{Proof of Theorem} \ref{new-7}

By \cite[Theorem 4.5 and Remark 4.6]{BC}, there is a $W$ invariant element $\delta \in \k[h]$ such that $H_{t,c}(h, W)\delta^{-1} \simeq \mathcal{D}(h^{reg})*W$.  $e \mathcal{D}_c(h_{reg})*We \simeq \mathcal{D}_c(h_{reg})^W$, hence $U_{t,c}(h,W)\delta^{-1} \simeq \mathcal{D}_c(h_{reg})^W$. $W$ acts freely  in $h_{req}$; hence may then use Theorem \ref{free}. We have $\mathcal{D}_c(h_{reg})^W \simeq  \mathcal{D}_c(h_{reg}
/W)$.

By assumption, $W$ is a pseudo-reflection group and $|W|$ is coprime to $char \, \k$. Chevalley-Shephard-Todd Theorem also hold in this situation (Theorem \ref{C-S-T}). Hence $h_{reg}/W$ is an affine rational variety and so by Theorem \ref{new-3}, the skew field of fractions indeed is $F_n(\k)$, $n=dim \, h$.
\end{proof}

\subsection{Contangent-bundle}

%\begin{theorem}\label{filtered2}
%Let $X$ be an smooth affine variety, and $G$ a finite group of automorphisms of it. It acts naturally on $\mathcal{O}(T^*X)$ by Poisson automorphisms. $\mathcal{D}_c(X)^G$, with the filtration of Proposition \ref{filtered}, is a filtered quantization of $\mathcal{O}(T^*X)^G$.
%\end{theorem}
%\begin{proof}
%Using Proposition \ref{filtered}, this result follows in the same way as the analogous result for usual rings of differential operators in characteristic $0$ (see, e. g., \cite{Schedler}).    
%\end{proof}

\begin{theorem}\label{quasi-classical1}
Let $X$ be an smooth affine variety, $G$ a finite group that acts freely on it. Then $T*(X/G)$ is isomorphic, as a Poisson variety, to $(T*X)/G$.    
\end{theorem}
\begin{proof}
It follows the same steps as Theorem \ref{free}, but simpler, since we are now at the level of contagente bundles  
\end{proof}

\begin{remark}
In characteristic $0$ the above result was shown in \cite[Theorem 3.4]{Schwarz}.    
\end{remark}

\begin{proposition}\label{quasi-classical2}

Let $X$ and $Y$ be birationally equivalent smooth affine varieties. Then $\k(T^*X)$ and $\k(T^* Y)$ are isomorphic as Poisson fields. 
\end{proposition}
\begin{proof}
The proof of \cite[Proposition 3.12]{Schwarz} works on any characteristic.
\end{proof}

\textbf{proof of Theorem}  \ref{new-6}

Let $U$ be an affine open subset of $X$ where $G$ acts freely. As $\k(X)=\k(U)$, $(1) \, Frac \, \mathcal{O}(T^*X)^G \simeq Frac \, \mathcal{O}(T^*U)^G$, as Poisson fields, by Proposition \ref{quasi-classical2}. By Theorem \ref{quasi-classical1}, $(2) \, Frac \, \mathcal{O}(T^*U)^G \simeq Frac \, \mathcal{O}(T^*(U/G))$ as Poisson fields. As $X/G$ is birational to $Y$, $U/G$ is also, and hence $(3) \, Frac \, \mathcal{O}(T^*(U/G)) \simeq \mathcal{O}(T^*Y)$ as Poisson fields, by Proposition \ref{quasi-classical2} again. Combining (1), (2), (3), we obtain our result.

Now we will use this theorem for Poisson-Noether's Problem. It is convenient to recall Chevalley-Shephard-Todd-(Serre) Theorem:

\begin{theorem}\label{C-S-T}
    Let $V$ be a finite dimensional vector space, $G<GL(V)$ a finite group whose order is not divisible by $char \, \k$. Then $S(V^*)^G$ is again a polynomial algebra if and only if $G$ is a pseudo-reflection group.
\end{theorem}
\begin{proof}
    \cite{Bourbaki}.
\end{proof}

If the characteristic of the field divides the order of the group, Serre (\cite{Serre}) has shown that for $S(V^*)^G$ to be a polynomial algebra, a necessary condition is that $G$ is a pseudo-reflection group; but it is not sufficient (see a counter-example in \cite[19-2]{Kane})

In characteristic 0, the pseudo-reflection groups have long been classified using the classification of  Shephard and Todd for complex reflection groups and Clark-Ewing Theorem (see, e.g., \cite[Chapter 15]{Kane}).

More recently the irreducible pseudo-reflection groups were classified over any characteristic,
by Kantor, Wagner, Zaleskii and Serezkin, and those for which the invariants of the polynomial algebra ara again polynomial by Kemper and Malle (see both aspects of the classification in \cite{KM}).

Let $\k$ be an algebraically closed field of prime characteristic, and $G$ an irreducible pseudo-reflection group. We will call the group K-M if it is in Kemper and Maller list.

\begin{theorem}\label{prime-Baudry}
    Let $G < GL(h)$ be a K-M group. Then if we make $G$ act diagonally on $h \oplus h^*$, with its canonical symplectic form, $(h \oplus h^*)/G$ is Poisson rational.
\end{theorem}
\begin{proof}
Like our work in the $\mathbb{C}$ case (\cite{Schwarz}) we may use Theorem \ref{new-6}. In our situation, $X=h$. $h/G$ is clearly birationally equivalent to $Y=h$ --- they are in fact isomorphic. Finally, recall that $T^*h=h \oplus h^*$. Apllying Theorem \ref{new-6} we obtain our desired result.   
\end{proof}
\section*{Appendix}

In the paper \cite{EG}, Etingof and Ginzburg conjectured (see Proposition 17.6* of their paper) an analogue of the Gelfand-Kirillov Conjecture for all spherical subalgebras (at $t=1$) of symplectic reflection algebras. Namely, Let $A_1(V)$ be the Weyl algebra of the symplectic vector space $(V, \omega)$ at $t=1$, and $\Gamma$ the symplectic reflection group. Let $n=dim \, V$. Then the skew field of fractions should be $F_n(\mathbb{C})^\Gamma$. When $\Gamma$ is a complex reflection group, the conjecture is true by the Dunkl embedding, and the skew field of fractions is in fact isomorphic to $F_n(\mathbb{C})$ (\cite{EFOS}). The other non-expceptional family of symplectic reflection groups on Cohen's classification \cite{Cohen} are the so called wreath product type $G \wr S_n$, $G$ a finite subgroup of $SL_2(\mathbb{C})$, which are classified \cite{Springer}. If $\Gamma=G \wr S_n$, it is folkore that $F_n(\mathbb{C})^\Gamma \simeq F_n(\mathbb{C})$. We offer a proof of this fact. First notice that if $G$ is a group of automorphisms of a ring $R$ and $H$ a normal subgroup of $G$, $R^G=(R^H)^{G/H}$. So when $\Gamma$ is of wreath product type, we have that $F_n(\mathbb{C})^\Gamma=((F_1(\mathbb{C})^G)^{\otimes n})^{S_n}$. By a result from \cite{AlevX}, $F_1(\mathbb{C})^G \simeq F_1(\mathbb{C})$. Hence $F_n(\mathbb{C})^\Gamma \simeq F_n(\mathbb{C})^{S_n} \simeq F_n(\mathbb{C})$. In the last isomorphism we used \cite[Theorem 4.1]{FMO}.

\section*{Acknowledgments}

The author would like the expresses his grattitude to V. Futorny, whom he had a lot of conversations about this paper, and R. Bezrukavnikov, with whom he discussed the proof of the essential Theorem \ref{free}. Finally, the author would like to thanks the love of his life, P. Marcondes, for their support in everything.


\begin{thebibliography}{9}

\bibitem{AlevX} J. Alev, F. Dumas, Invariants du corps de Weyl sous l’action de groupes finis, Commun. Algebra 25 (1997), 1655–1672.

\bibitem{Alev0} Alev, J.; Dumas, F. , Sur les invariants des algbres de Weyl et de leurs corps de fractions.
Rings, Hopf algebras, and Brauer groups (Antwerp/Brussels, 1996), 1–10, Lecture Notes in
Pure and Appl. Math., 197, Dekker, New York, 1998.

\bibitem{Alev} Alev, J.; Dumas, F. Opératours différentiels invariants er problemè de Noether, in: Studies in Lie Theory, in: Progre. Math., volume 243, Birkhäuser Boston, Boston, MA, 2006, pp. 21-50.J. 

\bibitem{AP} Alev and P. Polo, A rigidity theorem for finite group actions on enveloping
algebras of semisimple Lie algebras, Adv. Math. 111 (1995), no.
2, 208–226.

\bibitem{Baudry} J. Baudry. Structures de Poisson de certaines variétés quotients: propríetés homologiques, d’engendrement fini et de rationalité, Thèse de Doctorat, Reims 2009.

\bibitem{Bavula} Bavula, V. "Dimension, multiplicity, holonomic modules, and an analogue of the inequality of Bernstein for rings of differential operators in prime characteristic." Representation Theory of the American Mathematical Society 13.10 (2009): 182-227.
\bibitem{BO} Berthelot, P.; Ogus, A. Notes on Crystalline Cohomology, Mathematical Notes 21, Princeton University Press, Princeton, 1978.

\bibitem{BO} V. Bavula and F. van Oystayen, Simple holonomic modules over rings of
diferential operators with regular coeficients of Krull dimension 2, Trans-
actions of the American Mathematical Society, v. 353 (2001), 2193-2214.

\bibitem{BMR} Bezrukavnikov, R.; Mirković, I.; Rumynin, D.
Localization of modules for a semisimple Lie algebra in prime characteristic.
With an appendix by Bezrukavnikov and Simon Riche.
Ann. of Math. (2) 167 (2008), no. 3, 945–991.

\bibitem{Bohning} Christian Böhning, The Rationality Problem in Invariant Theory, arXiv:0904.0899



\bibitem{Bourbaki} Bourbaki, N.; Lie groups and Lie algebras. Chapters 4–6. Translated from the 1968 French original by Andrew Pressley. Elements of Mathematics (Berlin). Springer-Verlag, Berlin, 2002.


\bibitem{BC} Brown, K. A.; Changtong, K. Symplectic reflection algebras in positive characteristic. Proc. Edinb. Math. Soc. (2) 53 (2010), no. 1, 61–81.

\bibitem{Goodearl} Brown, Ken, and Ken R. Goodearl. Lectures on algebraic quantum groups. Birkhäuser, 2012.

\bibitem{Burnside} Burnside, W Theory of groups of finite order, second edition, Cambridge university press, 1911.


\bibitem{CH} Cannings, R.; Holland, M. P. Differential operators on varieties with a quotient
subvariety, J. Algebra 170 (3) (1994) 735–753

\bibitem{Cohen} Cohen, Arjeh M. "Finite quaternionic reflection groups." Journal of Algebra 64.2 (1980): 293-324.

\bibitem{C-T} Colliot-Thélène, J-L; Sansuc, J- The rationality problem for fields of invariants under linear algebraic groups (with special regards to the Brauer group). Algebraic groups and homogeneous spaces, 113–186, Tata Inst. Fund. Res. Stud. Math., 19, Tata Inst. Fund. Res., Mumbai, 2007.

\bibitem{Drensky} Drensky, V.; Formanek, E.
Polynomial identity rings.
Advanced Courses in Mathematics. CRM Barcelona. Birkhäuser Verlag, Basel, 2004.

\bibitem{Dolgachev} Dolgachev, I. V. Rationality of fields of invariants. Algebraic geometry, Bowdoin, 1985 (Brunswick, Maine, 1985), 3–16, Proc. Sympos. Pure Math., 46, Part 2, Amer. Math. Soc., Providence, RI, 1987.

\bibitem{Dumas} F. Dumas, An introduction to noncommutative polynomial invariants,
rédaction d'un cours présenté lors de l'école internationale CIMPA "Homological methods and representations of non-commutative algebras" (Mar del Plata, 2006). https://lmbp.uca.fr/$\sim$fdumas/recherche.html

\bibitem{EFOS} Eshmatov, F.; Futorny, V.; Ovsienko,S.; Schwarz, J. F. Noncommutative Noether's problem for complex reflection groups.

\bibitem{Formanek} FORMANEK, E. Rational function fields. Noether's problem and related questions. Journal of Pure and Applied Algebra, v. 31, n. 1-3, p. 28-36, 1984.

\bibitem{FH} Futorny, Vyacheslav, and Jonas T. Hartwig. "Solution of a q q-difference Noether problem and the quantum Gelfand–Kirillov conjecture for gl N." Mathematische Zeitschrift 276 (2014): 1-37.

\bibitem{FSS}
FUTORNY, Vyacheslav; SCHWARZ, Joao; SHESTAKOV, Ivan. LD-stability for Goldie rings. Journal of Pure and Applied Algebra, v. 225, n. 11, p. 106741, 2021.

\bibitem{Fryer} FRYER, Siân. The q-division ring and its fixed rings. Journal of Algebra, v. 402, p. 358-378, 2014.
\bibitem{FMO} Futorny, V; Molev, A; Ovsienko, S. The Gelfand-Kirillov conjecture and Gelfand-Tsetlin modules for finite W-algebras. Adv. Math. 223 (2010), no. 3, 773–796.

\bibitem{EG} Etingof, P.; Ginzburg, V. Symplectic reflection algebras, Calogero-Moser space, and deformed Harish-Chandra homomorphism. Invent. Math. 147 (2002), no. 2, 243–348.
\bibitem{FS}  Futorny, V.; Schwarz, J. Noncommutative Noether's problem vs classic Noether's problem. Math. Z. 295 (2020), no. 3-4, 1323–1335.

\bibitem{GK} I. M. Gelfand and K. K. Kirillov. Sur les corps liés aux algèbres envoloppantes des algèbres de Lie, Inst. Hautes Études Sci. Publ. Math. No. 31 (1966), 5-19.

\bibitem{Grothendieck} Grothendieck, A. Éléments de géométrie algébrique. IV. Étude locale des schémas et des morphismes de schémas IV. (French) Inst. Hautes Études Sci. Publ. Math. No. 32 (1967), 361 pp.

\bibitem{H} Hartwig, Jonas T. "The q-difference Noether problem for complex reflection groups and quantum OGZ algebras." Communications in Algebra 45.3 (2017): 1166-1176.

\bibitem{Hoshi} A. Hoshi, Noether's problem and rationality problem for multiplicative invariant fields: a survey, arXiv:2010.01517.

\bibitem{Jacobson} N. Jacobson, Basic Algebra II, (W. H. Freeman, San Francisco, 1980).

\bibitem{Erich} Jauch, Erich C. "An extension of U (gln) related to the alternating group and Galois orders." Journal of Algebra 569 (2021): 568-594.

\bibitem{Jensen} Jensen, C. U.; Ledet, A.; Yui, N. Generic polynomials. Constructive aspects of the inverse Galois problem. Mathematical Sciences Research Institute Publications, 45. Cambridge University Press, Cambridge, 2002.

\bibitem{Kane} Kane, Richard M. Reflection groups and invariant theory. Vol. 5. New York: Springer, 2001.

\bibitem{KM} Kemper, G.;  Malle, G. The finite irreducible linear groups with polynomial ring of invariants, Transformation Groups 2 (1997) 57–89.


\bibitem{KL} G. R. Krause and T. H. Lenegan. Growth of Algebras and Gelfand-Kirillov Dimension, Graduates Studies in Mathematics 22, American Mathematical Society, revised edition, 2000.

\bibitem{KPPV} KLEP, Igor et al. Noncommutative rational functions invariant under the action of a finite solvable group. Journal of Mathematical Analysis and Applications, v. 490, n. 2, p. 124341, 2020.

\bibitem{Lenstra} H. W. Lenstra, Rational functions invariant under a finite abelian group, Invent.
Math. 25 (1974), 299–325. 


\bibitem{Lorenz0} M. Lorenz, Regularity of multiplicative invariants, Comm. Algebra 1996, no. 3, 1051-1055.


\bibitem{Lorenz} Lorenz, M.
Multiplicative invariant theory.
Encyclopaedia of Mathematical Sciences, 135. Invariant Theory and Algebraic Transformation Groups, VI. Springer-Verlag, Berlin, 2005.

\bibitem{FS2} V. Futorny, J. Schwarz. Algebras of invariant differential operators. J. Algebra 542 (2020),
215-229.


\bibitem{McConnell} McConnell, J. C.; Robson, J. C. Noncommutative Noetherian rings. With the cooperation of L. W. Small. Revised edition. Graduate Studies in Mathematics, 30. American Mathematical Society, Providence, RI, 2001.

\bibitem{Miyata} Miyata, Takehiko. "Invariants of Certain Groups I1." Nagoya Mathematical Journal 41 (1971): 69-73.

\bibitem{Muhasky} A. J. Muhasky, The differential operator ring of an affine curve, Trans. of the American Math. Society 307 (1988), 705-723.


\bibitem{Mumford} Mumford, D,; Abelian varieties. Tata Institute of Fundamental Research Studies in Mathematics, 5. Published for the Tata Institute of Fundamental Research, Bombay by Oxford University Press, London, 1970

\bibitem{MS} Alapan Mukhopadhyay, Karen E. Smith, Some algebras with trivial rings of differential operators, arXiv:2404.09184.

\bibitem{Noether0} E. Noether, Rational Funktionkörper, Jahrbericht Deutsch, Math.-Verein. 22 (1913) , 316-319

\bibitem{Noether} Noether, E. Gleichungen mit vorgeschriebener Gruppe, Math. Ann. 78 (1916), 221–229

\bibitem{McConnell} McConnell, J. C.; Robson, J. C. Noncommutative Noetherian rings. With the cooperation of L. W. Small. Revised edition. Graduate Studies in Mathematics, 30. American Mathematical Society, Providence, RI, 2001.

\bibitem{Plans0} B. Plans, On Noether’s problem for central extensions of symmetric and alternating groups, J. Algebra 321 (2009) 3704–3713.
\bibitem{Plans} B. Plans, On Noether’s rationality problem for cyclic groups over Q, Proc. Amer. Math. Soc. 145 (2017) 2407–2409.

\bibitem{Premet}
A. Premet, Modular Lie algebras and the Gelfand-Kirillov conjecture, Invent. Math. 181
(2010), no. 2, 395–420.

\bibitem{Procesi} Procesi, C. Non-commutative affine rings, Atti Accad. Naz. Lincei Rend. 
Cl. Sci. Fis. Mat. Natur. 8 (1967),239-255. 

\bibitem{Revoy} Revoy, P.
Algèbres de Weyl en caractéristique p. (French)
C. R. Acad. Sci. Paris Sér. A-B 276 (1973), A225–A228.

\bibitem{Saltman} D. J. Saltman, Generic Galois extensions and problems in field theory, Adv. Math. 43
(1982), 250–283.

\bibitem{Schedler} Schedler, T. Deformations of algebras in noncommutative geometry, in: Noncommutative Algebraic Geometry, in: Math. Sci. Res. Inst. Publ., vol. 64, Cambridge
Univ. Press, New York, 2016, pp. 71–166.

\bibitem{Schwarz} Schwarz, J. A Poisson Noether's Problem and Poisson rationality, J. Algebra 606 (2022), 195-208.

\bibitem{Serre} J. P.-Serre, Groupes finis d'automorphisms d'anneaux locaux reguliers, Colloq. d'Alg Ecole Norm. de Jeunes Filles (1967), 1-11.

\bibitem{Smith0} S. P. Smith, Differential operators on commutative algebras, in Ring Theory, Lecture Notes
in Mathematics 1197, (1985), 164-177. 

\bibitem{Smith} Smith, S. P. Differential operators on the affine and projective lines in characteristic p>0. Séminaire d'algèbre Paul Dubreil et Marie-Paule Malliavin, 37ème année (Paris, 1985), 157–177, Lecture Notes in Math., 1220, Springer, Berlin, 1986.

\bibitem{Springer} T. A. Springer, Invariant theory, Lecture Notes in Mathematics, Vol. 585, Springer-Verlag,
Berlin-New York, 1977.

\bibitem{Swan} R. G. Swan, Invariant rational functions and a problem of Steenrod, Invent. Math. 7
(1969), 148–158.
\bibitem{Tikaradze} Tikaradze, A.
The noncommutative Noether's problem is almost equivalent to the classical Noether's problem.
Adv. Math. 396 (2022), Paper No. 108161, 4 pp.

\bibitem{Vokresenskii} V. E. Voskresenskii, On the question of the structure of the subfield of invariants of a
cyclic group of automorphisms of the field Q(x1 , . . . , xn) (russian), Izv. Akad. Nauk
SSSR ser. Mat. 34 (1970), 366–375. (English translation in Math. USSR-Izv. 8,4
(1970), 371–380.)

\bibitem{Zariski} O. Zariski, On Castelunovo's criterion of rationality pa=P2-0 of an algebraic surface, Illinois J. Math, 2 (1958), 303-315.
\end{thebibliography}
\end{document}